\documentclass{amsart}

\usepackage[utf8x]{inputenc}
\usepackage{amsmath}
\usepackage{amssymb}
\usepackage{amsthm}
\usepackage{graphicx}
\usepackage{overpic}
\usepackage{subcaption}
\usepackage[square,numbers,sort]{natbib}
\usepackage{colonequals}
\usepackage{todonotes}
\usepackage{tikz}
\usepackage{hyperref}

\newcommand{\dist}{\operatorname{dist}}
\newcommand{\R}{\mathbb{R}}

\providecommand{\norm}[1]{\lVert#1\rVert}
\newcommand{\parder}[2]{\frac{\partial #1}{\partial #2}}

\DeclareMathOperator*{\argmin}{arg\,min}

\newtheorem{definition}{Definition}[section]
\newtheorem{theorem}{Theorem}[section]
\newtheorem{lemma}{Lemma}[section]

\graphicspath{{gfx/}}

\title{Test function spaces for geometric finite elements}
\author{Oliver Sander}

\begin{document}

\begin{abstract}
We construct test function spaces for geometric finite elements.  Geometric finite elements (GFE) are generalizations
of Lagrangian finite elements to situations where the unknown function maps into a nonlinear space.  Test functions
for such spaces arise as variations of
GFE functions wherever the GFE function space has a local manifold structure.
For any given GFE function $u_h$, the test functions form a linear space that depends on $u_h$.
They generalize Jacobi fields in the same way that the GFE interpolation functions generalize
geodesic curves.  Having test function spaces allows to extend the
GFE method to boundary value problems that do not have a minimization formulation.

\end{abstract}

\maketitle

\section{Introduction}

Geometric finite elements (GFE) are generalizations of Lagrangian finite elements that discretize spaces of maps
into a nonlinear Riemannian manifold $M$.  The classical Lagrangian finite element functions are recovered when $M = \R$.
The generalizations can be achieved in a number of ways.  Originally,
{\em geodesic} finite elements were introduced for one-dimensional, first-order approximations~\cite{sander:2010}, and subsequently
generalized to domains of arbitrary dimension~\cite{sander:2012}, and higher approximation orders~\cite{sander:2015}.
Optimal discretization error bounds were proved in~\cite{grohs_hardering_sander:2014, hardering:2015}, and
the discretization has been applied
successfully to problems in Cosserat mechanics~\cite{sander_neff_birsan:2016,sander:2010} and the computations
of harmonic maps~\cite{sander:2015}.
Later, {\em projection-based} finite elements were proposed and investigated in~\cite{grohs_sprecher:2013,sprecher:2016}.

In all these publications, geometric finite elements have only been applied to problems with a minimization formulation.  The fact that
GFE functions are $H^1$ maps allowed to reformulate the energy formulations straightforwardly as minimization problems
for algebraic functionals defined on a product manifold $M^n$, with $n$ the number of Lagrange nodes of
the grid.  Variations and optimality were only considered in this algebraic setting.

However, even though this does not directly follow from the original publications, GFE can also be used for problems
without a minimization structure.  The missing ingredient for this are spaces of suitable test functions.  In this
short note we construct such spaces.  Their definition follows directly from the local manifold structure of the discrete space,
and they generalize the Jacobi fields of classical differential geometry (e.g., \cite[Chap.\,5]{jost:2011}).
As such, they form linear spaces
of vector fields along GFE functions, and these spaces can be identified with their values at the Lagrange nodes.
Disregarding a few minor technical differences, the construction is the same both for geodesic finite elements
and for projection-based finite elements.

The construction of test functions as variations will appear trivial to people with experience
in geometric analysis.  On the other hand, for people with a numerical analysis background this construction
may not be quite as clear,
and the author has been prompted to write this article by repeated questions about the existence and
nature of GFE test functions.

Evaluating GFE test functions is easy and cheap.  For geodesic finite elements, provided the value of a GFE function
$u_h : \Omega \to M$ is given at a point $x$, then evaluating a test function for $u_h$ at the same point involves only
solving one linear system of equations in $\dim M$ variables.  For projection-based finite elements,
the evaluation procedure depends
on how the projection onto $M$ from an embedding space can be computed.  For the important case of $M$ being
a sphere, there is even a closed-form expression for the test functions.

Having test functions allows to state optimality conditions for minimization problems directly in
GFE spaces, in contrast to the approach of~\cite{sander:2012,sander:2015}, which formulated optimality conditions
only in the algebraic setting.  We work out both approaches, and show in Section~\ref{sec:minimization}
by trivial computations that both are equivalent.
This otherwise obvious result justifies our construction.

\section{Geometric finite elements}

We briefly review the two main constructions of geometric finite elements.  These differ only in the way
Lagrange interpolation on a single element is generalized to nonlinear spaces.  The first approach,
geodesic interpolation, is completely intrinsic.  Alternatively, projection-based interpolation needs
an embedding space of $M$, but leads to more efficient algorithms for certain choices of $M$.

\subsection{Sobolev spaces of manifold-valued functions}

Let $\Omega$ be an open and connected subset of $\R^d$ with a Lipschitz boundary,
and let $M$ be a smooth, connected manifold.
The following definition of a Sobolev space for functions with values in $M$ is standard
(see, e.g., \cite{schoen_uhlenbeck:1982,hardering:2015}).
\begin{definition}
\label{def:sobolev_space}
 Let $\imath:M\to \mathbb{R}^N$ be an isometric embedding (which always exists by \cite{nash:1956}).
 For $k \in \mathbb{N}_0$ and $p \in \mathbb{N} \cup \{\infty\}$ define
 \begin{equation*}
	W^{k,p}(\Omega,M)\colonequals
	\left\{v\in W^{k,p}(\Omega,\mathbb{R}^N):\, v(x)\in \imath(M),\ a.e.\right\}.
 \end{equation*}
\end{definition}

For nonlinear $M$ these spaces obviously do not form vector spaces.  However, under certain
smoothness conditions the manifold structure of $M$ is inherited.  The following result
is proved in~\cite{palais:1968}.

\begin{lemma}
If $k>d/p$, the spaces $W^{k,p}(\Omega,M)$ are Banach manifolds.
\end{lemma}

The GFE method is a way to discretize such nonlinear function spaces.  Its central idea are generalizations of Lagrange interpolation
to interpolation of values given on a nonlinear manifold.

\subsection{Geodesic interpolation}
\label{sec:geodesic_interpolation}

\begin{figure}
 \begin{center}
 \begin{overpic}[width=0.7\textwidth]{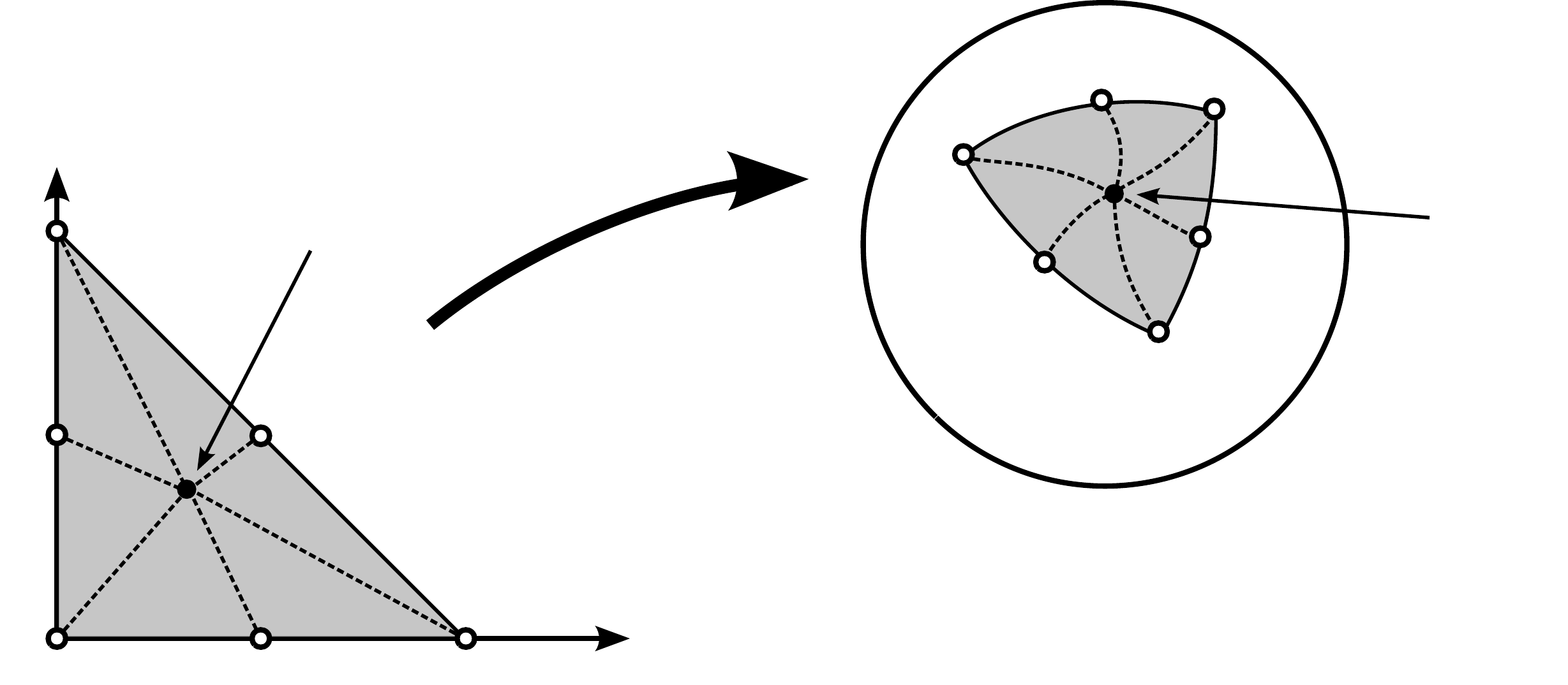}
  \put( 1, 0){$a_1$}
  \put(15, 0){$a_2$}
  \put(28, 0){$a_3$}
  \put(-1,15){$a_4$}
  \put(18,17){$a_5$}
  \put( 5,29){$a_6$}
  \put(75,20){$v_1$}
  \put(78,27){$v_2$}
  \put(79,36){$v_3$}
  \put(64,24){$v_4$}
  \put(69,39){$v_5$}
  \put(57,33){$v_6$}
  \put(20,29){$\xi$}
  \put(35,31){$\Upsilon^p$}
  \put(85,15){$M$}
  \put(92,28){$\Upsilon^\text{ge}(v,\xi)$}
 \end{overpic}
 \end{center}
 \caption{Second-order geodesic interpolation from the reference triangle into a sphere}
 \label{fig:geodesic_interpolation}
\end{figure}
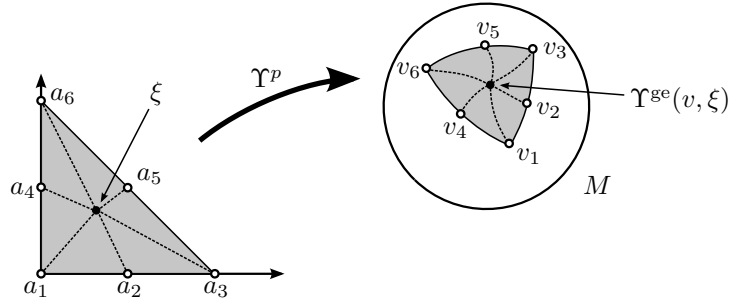

The first approach to interpolation of values on $M$ is completely intrinsic.
Let $T_\text{ref}$ be a bounded domain in $\R^d$, with coordinates $\xi$.
We call $T_\text{ref}$ a reference element.
On its closure $\overline{T_\text{ref}}$ we have a set of
distinct Lagrange nodes $a_i, 1\le i \le m$, and corresponding scalar Lagrangian interpolation
functions $\varphi_1,\dots,\varphi_m$, i.e., $p$-th order polynomial functions with $\varphi_i(a_j) = \delta_{ij}$.
We assume that the $\varphi_i$ and $a_i$ are such that the corresponding interpolation
problem is well posed, i.e., for given $v_i \in \R$, $i=1,\dots,m$ there is a single function $\pi : T_\text{ref} \to \R$
in the span of the $\varphi_i$ such that $\pi(a_i) = v_i$ for all $1\le i \le m$.

We want to construct a function $\Upsilon^\text{ge} : T_\text{ref} \to M$ that interpolates a given set of values
$v_1,\dots,v_m \in M$.
The following definition has been given in~\cite{sander:2015} and \cite{grohs:2011}. It is visualized in Figure~\ref{fig:geodesic_interpolation}.
\begin{definition}
\label{def:geodesic_interpolation}
 Let $T_\text{ref} \subset \R^d$ be a bounded domain,
 and $M$ a connected, smooth, complete manifold with a distance metric $\dist(\cdot,\cdot) : M \times M \to \R$.
 Let $\varphi_1,\dots,\varphi_m$ be a set of  $p$-th order
 scalar Lagrangian shape functions, and let $v = (v_1,\dots,v_m) \in M^m$ be values at the
 corresponding Lagrange nodes.  We call
\begin{align*}
 \Upsilon^\text{ge} & \; : \; M^m \times T_\text{ref} \to M \\
 \Upsilon^\text{ge}(v,\xi) & \colonequals \argmin_{q \in M} \sum_{i=1}^m \varphi_i(\xi) \dist(v_i,q)^2
\end{align*}
$p$-th order geodesic interpolation on $M$.
For given $T_\text{ref}$, $a_1,\dots,a_m$, and $\varphi_1,\dots,\varphi_m$, the space of all such functions will be denoted by
$P_p^\text{ge}(M)$.
\end{definition}
We will look at the functions $\Upsilon^\text{ge}$ sometimes as functions of $\xi$ or of the $v_1,\dots,v_m$ only, and adapt
our notation accordingly.  The meaning should always be clear from the context.

As values of $\Upsilon^\text{ge}$ are minimizers of a functional
\begin{equation}
\label{eq:gfe_potential}
 f_{v,\xi} (q) \colonequals \sum_{i=1}^m \varphi(\xi)_i \dist(v_i,q)^2,
\end{equation}
they fulfill a first-order optimality
criterion.  For any $q \in M$, let $\log_q$ be the inverse of the exponential map of $M$ at $q$.
Then, we have
\begin{equation}
\label{eq:first_order_optimality}
 \sum_{i=1}^m \varphi_i(\xi) \log_{\Upsilon^\text{ge}(v,\xi)} v_i = 0,
\end{equation}
provided the quantities are close enough to each other so that $\log$ is defined~\cite[Thm.\,1.2]{karcher:1977}.

\bigskip

Important properties, like $C^\infty$-differentiability of $\Upsilon^\text{ge}$ with respect to all of its arguments,
have been shown in~\cite{sander:2012,sander:2015}.  Of central
importance for the construction of test functions is the well-posedness of the definition, which holds
if the nodal values $v_1,\dots,v_m$
are ``close together'' in a certain sense. The precise conditions for first-order functions
have been given by Karcher \cite{karcher:1977}.
Remember that a set $D \subset M$ is called convex if for
each $p, q \in D$ the minimizing geodesic from $p$ to $q$
is entirely contained in~$D$.

\begin{theorem}[Karcher \cite{karcher:1977}]
 \label{thm:unique_riemannian_barycenter}
 Let $M$ be complete, $\varphi_1, \dots, \varphi_m$ a first-order scalar Lagrange basis on a reference
 element $T_\text{ref}$, $B_\rho$ an open convex geodesic ball of radius $\rho$ in $M$,
 and $v_1,\dots,v_m \in B_\rho$.
 \begin{enumerate}
  \item If the sectional curvatures of $M$ in $B_\rho$ are bounded by
     a positive constant $K$, and $\rho < \frac{1}{4} \pi K^{-1/2}$, then
     the function $f_{v,\xi}$ defined in \eqref{eq:gfe_potential} has a unique minimizer in $B_\rho$ for all $\xi \in T_\text{ref}$.
    \item
     If the sectional curvatures of $M$ in $B_\rho$ are at most $0$,
     then $f_{v,\xi}$ has a unique minimizer in $B_\rho$ for all $\xi \in T_\text{ref}$.
 \end{enumerate}
\end{theorem}

Different arguments are needed to show corresponding results for Lagrange polynomials $\varphi_i$ of order~2
and higher, because such polynomials can take negative values.  A simple proof for the following
qualitative result is given in \cite{hardering:2015}.  A more quantitative result appears in~\cite{sander:2015}.

\begin{theorem}
\label{thm:well_posedness}
 Let $M$ be complete, and $\varphi_1, \dots, \varphi_m$ a scalar Lagrange basis on $T_\text{ref}$.
 Let $B_D \subset B_\rho$
 be two concentric geodesic balls in $M$ of radii $D$ and $\rho$, respectively.
 There are numbers $D$ and $\rho$ such that if the values $v_1, \dots, v_m$ are
 contained in $B_D$, then the function $f_{v,\xi} (q)$ defined in \eqref{eq:gfe_potential}
 has a unique minimizer in $B_\rho$.
\end{theorem}

\subsection{Projection-based interpolation}

An alternative generalization of Lagrange interpolation uses embedding spaces and projections.
Let again $\imath : M \to \R^N$ be an isometric embedding, and let $\mathcal{P} : \R^N \supset U \to M$
be a projection from a neighborhood $U$ of $\imath(M)$ onto $M$.  Let $T_\text{ref}$, $a_i$,
$\varphi_i$, and $v_i$, $i=1,\dots,m$, be as in the previous section.  Then we can interpolate between the values $v_i$
by first interpolating in the embedding space, and then projecting pointwise onto $M$.
\begin{definition}[\cite{grohs_sprecher:2013,sprecher:2016}]
\label{def:projection_based_interpolation}
 Let $T_\text{ref} \subset \R^d$ be a bounded domain,
 $M$ a connected manifold with an embedding $\imath : M \to \R^N$, and let $\mathcal{P}$ be a projection
 from a neighborhood $U$ of $\imath(M)$ onto $M$.
 Let $\varphi_1,\dots,\varphi_m$ be a set of  $p$-th order
 scalar Lagrangian shape functions, and let $v = (v_1,\dots,v_m) \in M^m$ be values at the
 corresponding Lagrange nodes.  We call
\begin{align}
 \nonumber
 \Upsilon^\text{pr} & \; : \; M^m \times T_\text{ref} \to M \\
 \label{eq:projection_based_interpolation}
 \Upsilon^\text{pr}(v,\xi) & \colonequals \mathcal{P}\Big[\sum_{i=1}^m \varphi_i(\xi) \imath(v_i)\Big]
\end{align}
projection-based interpolation on $M$.
For given $T_\text{ref}$ and $\varphi_1,\dots,\varphi_m$, the space of all such functions will be denoted by
$P_p^\text{pr}(M)$.
\end{definition}
Like geodesic interpolation, projection-based interpolation is usually not defined for all combinations of
values $v_1,\dots,v_m \in M$.  The reason is that it is frequently impossible to define continuous projections
$\mathcal{P}$ onto $M$
on all of $\R^N$.  The intuition that interpolation is well-defined if the $v_1,\dots,v_m$ are close to each
still holds, but no more precise results are currently available.

\citeauthor{sprecher:2016} showed in~\cite{sprecher:2016} that projection-based interpolation can be interpreted as geodesic interpolation
if $M$ is equipped with the distance metric of the Euclidean embedding space.
\begin{lemma}[{\cite[Prop.\,1.4.4]{sprecher:2016}}]
 Let $\mathcal{P}$ be the closest-point projection from $\R^N$ onto $M$.  With the notation of
 Definitions~\ref{def:geodesic_interpolation} and~\ref{def:projection_based_interpolation} we have
 \begin{equation*}
  \Upsilon^\text{pr}(v,\xi)
  \colonequals
  \mathcal{P}\Big[\sum_{i=1}^m \varphi_i(\xi) \imath(v_i)\Big]
  =
  \argmin_{q \in M} \sum_{i=1}^m \varphi_i(\xi) \norm{\imath(v_i) - \imath(q)}^2.
 \end{equation*}
\end{lemma}
Nevertheless, the direct definition~\eqref{eq:projection_based_interpolation} is advantageous if an
easily computable $\mathcal{P}$ is available.  In that case, \eqref{eq:projection_based_interpolation} is a
straightforward way
to compute the interpolation function, unlike the implicit construction used in Definition~\ref{def:geodesic_interpolation}.
Also, it is quite obvious from~\eqref{eq:projection_based_interpolation} that $\Upsilon^\text{pr}$
is differentiable whenever $\mathcal{P}$ is, which is more difficult to show for geodesic interpolation.

Useful embeddings are available for a number of important spaces.  If $M$ is the unit sphere,
$\mathcal{P}$ is easily evaluated as
\begin{equation*}
 \mathcal{P}(w) = \frac{w}{\norm{w}}.
\end{equation*}
The interpolation polynomial $\Upsilon^\text{pr}$ is then well-defined unless $\sum_{i=1}^m \varphi_i(\cdot) \imath(v_i)$
has a zero.  The derivative of $\Upsilon^\text{pr}$ with respect to $\xi$ is
\begin{align*}
 \frac{\partial}{\partial \xi}\Upsilon^\text{pr}(v,\xi)
 =
 \frac{\partial \mathcal{P}(w)}{\partial w}\bigg|_{w = \sum_{i=1}^m \varphi_i(\xi) \imath(v_i)}
 \cdot
 \sum_{i=1}^m \frac{\partial \varphi_i}{\partial \xi} \imath(v_i),
\end{align*}
where
\begin{equation*}
 \frac{\partial \mathcal{P}(w)}{\partial w}
 =
 I\norm{w}^{-1} - ww^T\norm{w}^{-3}.
\end{equation*}

For the special orthogonal group $\text{SO}(N) \subset \R^{N \times N}$, the closest-point projection in the
Frobenius norm is the polar decomposition~\cite{neff_lankeit_madeo:2014}.  Closed-form expressions for the
polar factor of a given matrix $A$ exist~\cite{jog:2002}, but it is more convenient to compute it and its derivatives through
the iteration defined by $Q_0\colonequals A$ and
\begin{equation*}
Q_{k+1}\colonequals \frac{1}{2}\left(Q_k+Q_k^{-T} \right),
\end{equation*}
which converges quadratically to the polar factor of $A$~\cite{higham:1986}.
Alternatively, one may consider different projections like Gram--Schmidt orthogonalization
or QR decomposition.

Some manifolds like the set of all symmetric positive definite $N \times N$ matrices form open subsets of Euclidean spaces.
In such a case no natural projection is available.

\subsection{Global geometric finite element spaces}

Given an interpolation rule from one of the two previous sections, it is easy to construct global finite element
spaces.  Most of what follows in the rest of this paper is independent of whether geodesic or projection-based
interpolation is used.  We write $\Upsilon$ to mean either one of $\Upsilon^\text{ge}$ and $\Upsilon^\text{pr}$,
and $P_p$ for the corresponding spaces of generalized polynomials.

Let $\Omega$ be an open bounded subset of $\R^d$, $d \ge 1$.
For simplicity we assume that $\Omega$ has a polygonal boundary.  Let $\mathcal{G}$
be a conforming grid for $\Omega$ with elements of arbitrary type.
We denote by $x_i \in \Omega$, $i=1,\dots,n$ the union of the sets of Lagrange nodes
of the individual elements.

\begin{definition}[Geometric finite elements~\cite{sander:2012,sander:2015,sprecher:2016}]
\label{def:geodesic_finite_elements}
We call $v_h : \Omega \to M$
a $p$-th order geometric finite element function if it is continuous, and if
for each element $T \in \mathcal{G}$ the restriction $v_h|_T$ is a $p$-th order
geometric interpolation in the sense that
\begin{equation*}
 v_h|_T(x) = \Upsilon \big(v_{T,1}, \dots, v_{T,m}; \mathcal{F}_T(x) \big),
\end{equation*}
where $\mathcal{F}_T : T \to T_\text{ref}$ is affine or multilinear, $T_\text{ref}$ is the
reference element corresponding to $T$, and the $v_{T,i}$ are values in $M$.
  The space of all such functions $v_h$ (for a fixed $p$) will be denoted by $V_h^M$.
\end{definition}

Geometric finite element functions are Sobolev functions in the sense of Definition~\ref{def:sobolev_space}.
The following conformity result was shown in \cite{sander:2012}.
\begin{theorem}
\label{thm:conformity}
 $V_h^M(\Omega) \subset H^1(\Omega,M).$
\end{theorem}

\bigskip

\begin{figure}
 \begin{center}
  \includegraphics[width=0.48\textwidth]{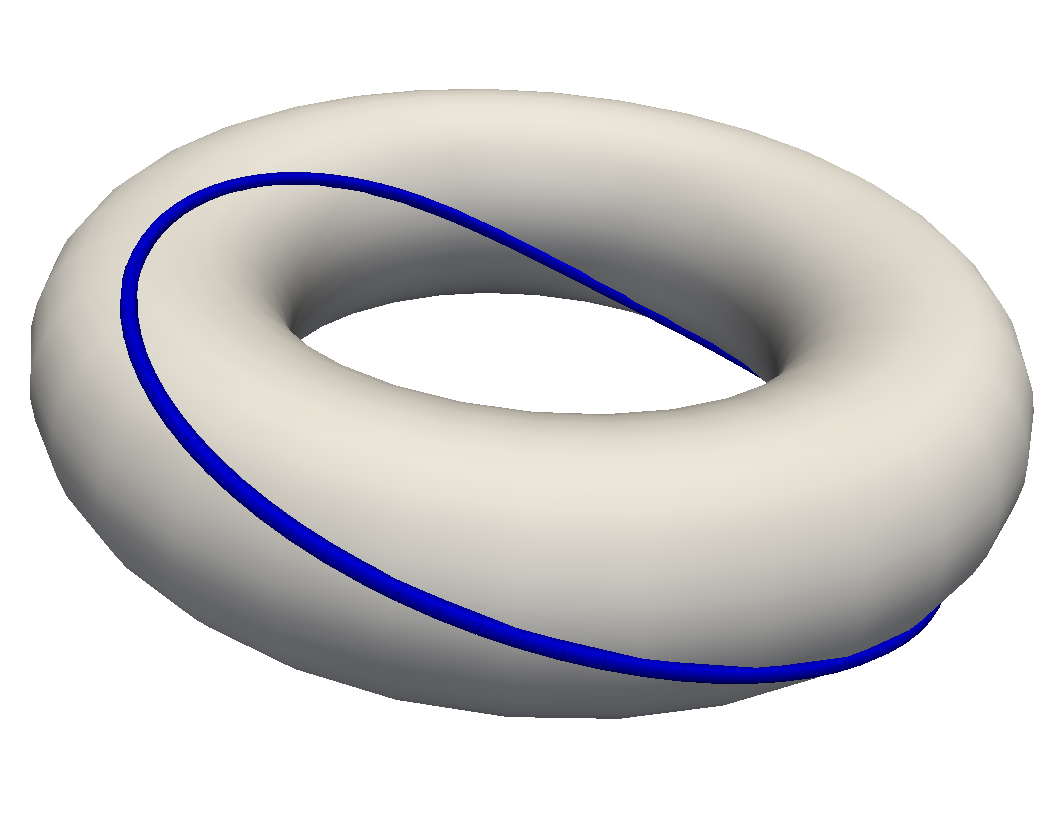}
 \end{center}
 \caption{Illustration of the global structure of the simple-most GFE space.  Let $\Omega = [0,1]$ be discretized
   by a single element, set $p=1$ and $M = S^1$.  Then the algebraic configuration space is $S^1 \times S^1$, which is isomorphic
   to a torus in $\R^3$.  The discrete space $V_h^{S^1}$ is isomorphic to that torus, except for all configurations
   $(v_1,v_2) \in S^1 \times S^1$ with $v_1 = -v_2$.    For each such configuration, there are two
   geodesic FE functions interpolating the values $v_1$ and $v_2$, and hence the torus has two ``sheets'' there.
   On the other hand, no projection-based interpolation exists for these configurations at all.
   The problematic configurations form a closed path which circles once
   around the torus, illustrated by the blue line.}
 \label{fig:gfe_space_is_not_a_manifold}
\end{figure}

We now briefly discuss aspects of the global structure of the set $V_h^M$.  Partial results can be obtained
by using its relationship to the product space $M^n \colonequals \prod_{i=1}^n M$ (the ``algebraic'' space).
However, the relationship between geometric finite element
functions $v_h \in V_h^M$ and sets of coefficients $\bar{v} \in M^n$ is more subtle than in the linear case,
where the two are isomorphic to each other.
The problem is that the coefficients may be such that the interpolation problem is not well posed on all
grid elements. We try to illustrate this effect in Figure~\ref{fig:gfe_space_is_not_a_manifold} for $M = S^1$,
$d=1$, $p=1$, with a grid $\mathcal{G}$ consisting of a single element.
Here, the algebraic space is the torus $S^1 \times S^1$, but the corresponding GFE space $V_h^{S^1}$
is larger if geodesic interpolation is used, and smaller for projection-based interpolation.

To formally investigate the relationship we define the nodal evaluation operator
\begin{align*}
 \mathcal{E} & \; : \; V_h^M \to M^n \\
 (\mathcal{E} (v_h))_i & = v_h(x_i),
 \qquad \text{$x_i$ the $i$-th Lagrange node of $\mathcal{G}$.}
\end{align*}
To each geometric finite element function $v_h \in V_h^M$ it associates the set of
function values at the Lagrange nodes.  Since functions in $V_h^M$ are continuous,
the operator $\mathcal{E}$ is well-defined and single-valued for all $v_h \in V_h^M$.

For traditional finite elements with values in a linear space, the evaluation operator is an isomorphism.
In particular, its inverse $\mathcal{E}^{-1}$, which associates finite element functions to a given
set of coefficients, exists everywhere, and is single-valued.  This does not hold for geometric finite elements.
For arbitrary $\bar{v} \in M^n$, the operator $\mathcal{E}^{-1}$ may be multi-valued,
or may not exist at all.  For geodesic interpolation, Theorem~\ref{thm:well_posedness} allows to characterize
the sets of coefficients for which $\mathcal{E}^{-1}$ is single-valued.
See \cite{sander:2015} for details.

However, the algorithmic treatment of finite element functions can only
work by manipulating an algebraic representation of finite element functions.  We therefore restrict
our attention to the set where $\mathcal{E}^{-1}$ is defined and single-valued
\begin{equation*}
 \widetilde{\mathcal{M}}
 \colonequals
 \big\{ \bar{v} \in M^n \; : \; \text{$\mathcal{E}^{-1}$ is defined and single-valued at $\bar{v}$} \big\}.
\end{equation*}
It is currently an open problem whether the set $\widetilde{\mathcal{M}}$ is open.
To be able to argue with manifold properties of the algebraic space, we restrict our attention
even further, to the interior of $\widetilde{\mathcal{M}}$
\begin{equation*}
 \mathcal{M}
 \colonequals
 \operatorname{int} \widetilde{\mathcal{M}}.
\end{equation*}
As an open subset of a manifold, $\mathcal{M}$ is a manifold itself.  It is unclear whether $\mathcal{M}$
is connected, but we will not use connectedness.
However, for geodesic interpolation the following theorem
shows that $\mathcal{M}$ has a nonempty interior, if the grid is fine enough.

\begin{theorem}[\cite{sander:2015}]
\label{thm:bijection_functions_coefficients_on_fine_grids}
Let $M$ be a Riemannian manifold, and let $v : \Omega \to M$ be Lipschitz continuous
in the sense that there exists a constant $L$ such that
\begin{equation*}
 \dist(v(x), v(y)) \le L\norm{x-y}
\end{equation*}
for all $x,y \in \Omega$.  Let $\mathcal{G}$ be a grid of $\Omega$ and $h$
the length of the longest edge of $\mathcal{G}$.  Let $\mathcal{E}^\text{ge}$ be the evaluation operator
of geodesic finite elements.  Let $\bar{v}$ be the set of value of $v$ at the Lagrange nodes.
For $h$ small enough, the inverse of $\mathcal{E}^\text{ge}$ has only a single value in $V_h^M$
for each $\tilde{v} \in M^n$ in a neighborhood of $\bar{v}$.
\end{theorem}

Similar results can be shown for projection-based interpolation.
They imply that for a given problem with a Lipschitz-continuous
solution we can always find a grid fine enough such
that we can disregard the distinction between $V_h^M$ and $M^n$ in the vicinity
of the solution.
Hence locally a geometric finite element problem can be represented by a corresponding
algebraic problem on the product manifold $M^n$.
In numerical experiments, this requirement of locality does not appear to pose a serious
obstacle.

Locally around functions where Theorem~\ref{thm:bijection_functions_coefficients_on_fine_grids} applies,
the function space $V_h^M$ inherits the differentiable manifold structure of
$\mathcal{M} \subset M^n$, because functions defined by geodesic or projection-based interpolation depend
differentiably on their corner values \cite[Thm.\,4.1]{sander:2015}.
In an abuse of notation, we will treat $V_h^M$ as a manifold itself.

\section{Test function spaces}
\label{sec:test_function_spaces}

To motivate our construction of test functions we briefly revisit the theory of linear elliptic partial differential equations.
Consider the linear reaction--diffusion equation
\begin{equation*}
 - \Delta u + u = f
 \qquad
 \text{on $\Omega$},
\end{equation*}
for a given function $f \in L^2(\Omega)$.  The weak formulation is
\begin{equation}
\label{eq:weak_reaction_diffusion}
 a(u,v) = (f,v)_{L^2(\Omega)},
\end{equation}
where
\begin{equation*}
 a(w,v)
 \colonequals
 \int_\Omega \nabla w \nabla v\,dx + \int_\Omega wv \,dx
\end{equation*}
is a bilinear form on $H^1(\Omega)$.  Of its two arguments, the second one is called a {\em test function}.

The test function $v$ is to be interpreted as a small variation around $w$.  Indeed, solutions of~\eqref{eq:weak_reaction_diffusion}
are minimizers of the functional
\begin{equation*}
 \mathcal{J} : H^1(\Omega) \to \R,
 \qquad
 \mathcal{J}(w) \colonequals \frac{1}{2} a(w,w) - (f,w).
\end{equation*}
Local minimizers of $\mathcal{J}$ are characterized by the directional derivative of $\mathcal{J}$ being zero in all directions.
The directional derivative of $\mathcal{J}$ at a point $w \in H^1$ in the direction of a function $v$ is
\begin{equation*}
 \frac{d\mathcal{J}(w)}{dv}
 =
 a(w,v) - (f,v)_{L^2(\Omega)}.
\end{equation*}
Hence the test function $v$ in~\eqref{eq:weak_reaction_diffusion} can be interpreted as a direction vector based at $u$.

If we now consider functionals defined on a manifold $\mathcal{N}$, then
small variations of a function $u \in \mathcal{N}$ are the elements of the tangent space $T_u \mathcal{N}$
of $\mathcal{N}$ at u.  There is not a single test function space anymore; rather, each configuration
$u \in \mathcal{N}$ has its own test function space.  Note, though, that for each $u \in \mathcal{N}$,
the corresponding test functions form a linear space.

This construction is no contradiction to the linear theory, which takes $u$ and $v$ from the same space $H^1(\Omega)$.
Indeed, if $u$ is element of the space $H^1(\Omega)$ (which, for the sake of the intuitive argument here, we interpret
as a manifold), then the test functions $v$ must be chosen from the tangent
space $T_u H^1(\Omega)$ of $H^1(\Omega)$ at $u$.  However, since $H^1(\Omega)$ is a linear space, all its tangent
spaces are isomorphic to the base space $H^1(\Omega)$.  Therefore, claiming that the test functions
must be from $H^1(\Omega)$ itself is merely an abuse of notation.

\subsection{Generalized Jacobi fields}

\begin{figure}
  \begin{minipage}{\textwidth}
   \begin{center}
    \includegraphics[width=0.39\textwidth, clip, trim= 150 25 150 70]{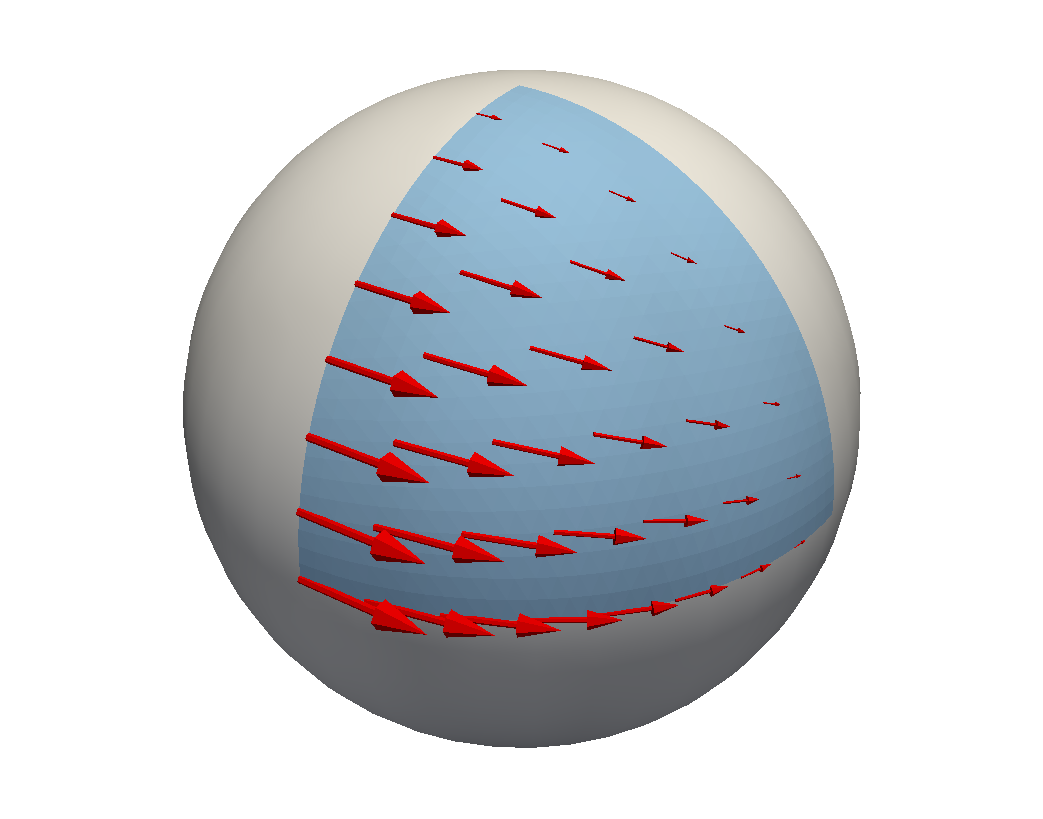}
    \hspace{0.03\textwidth}
    \includegraphics[width=0.39\textwidth, clip, trim= 150 25 150 70]{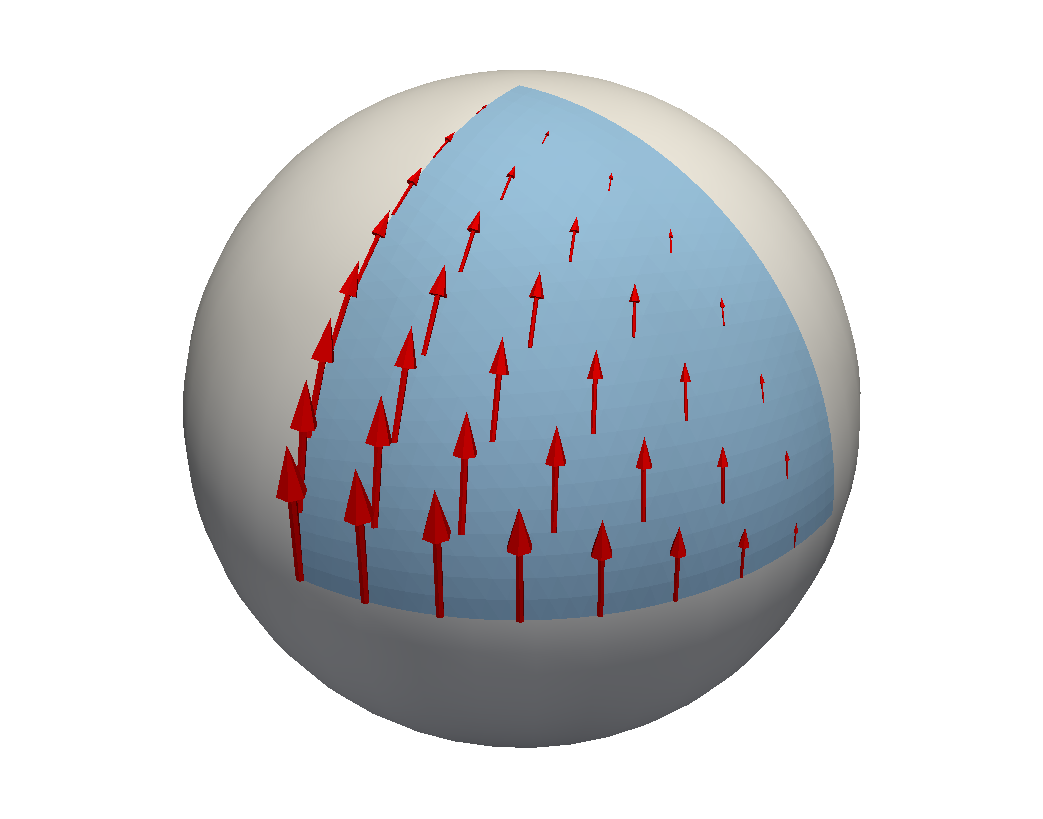}
   \subcaption{First-order}
   \end{center}
  \end{minipage}

  \vspace{0.02\textheight}

  \begin{minipage}{\textwidth}
   \begin{center}
    \includegraphics[width=0.39\textwidth, clip, trim= 150 25 150 70]{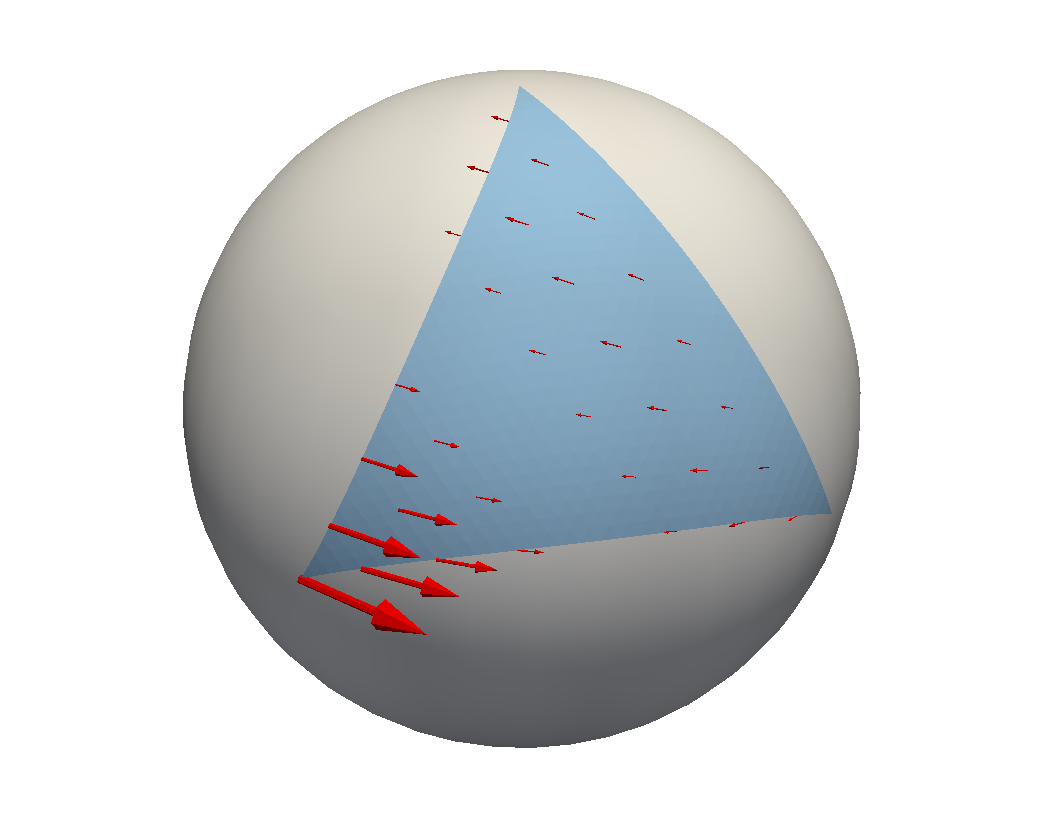}
    \hspace{0.03\textwidth}
    \includegraphics[width=0.39\textwidth, clip, trim= 150 25 150 70]{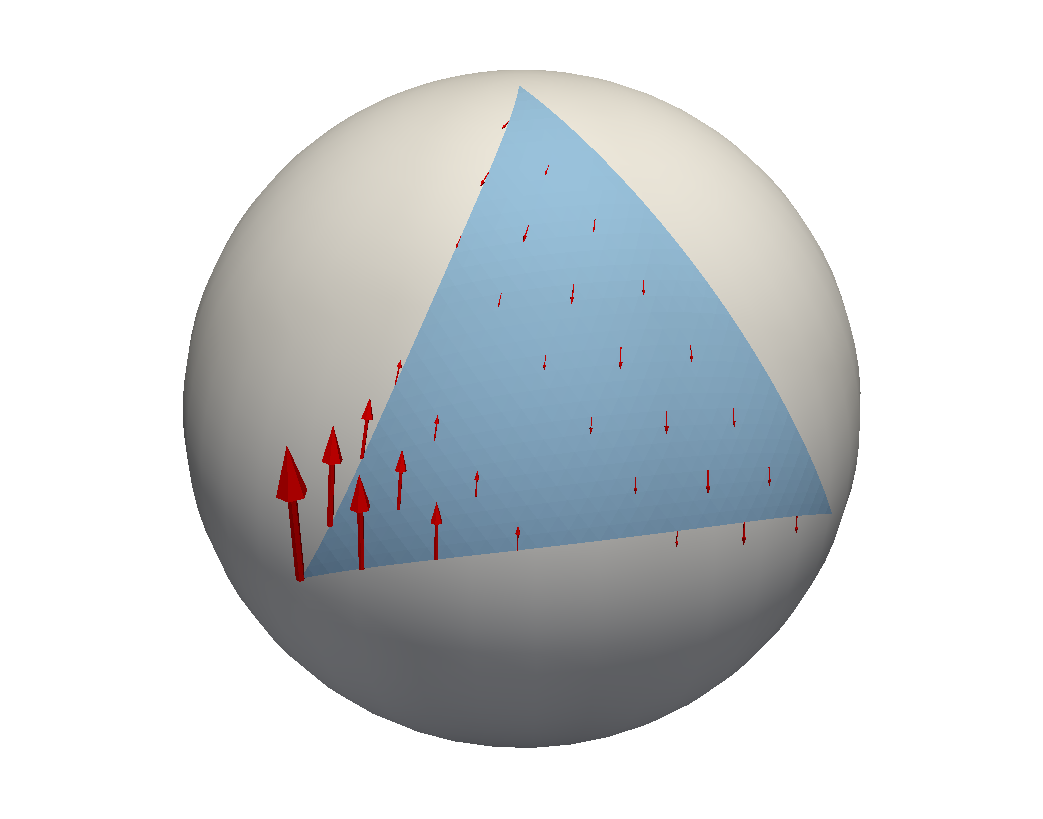}
   \subcaption{Second-order, vertex degree of freedom}
   \end{center}
  \end{minipage}

  \vspace{0.02\textheight}

  \begin{minipage}{\textwidth}
   \begin{center}
    \includegraphics[width=0.39\textwidth, clip, trim= 150 25 150 70]{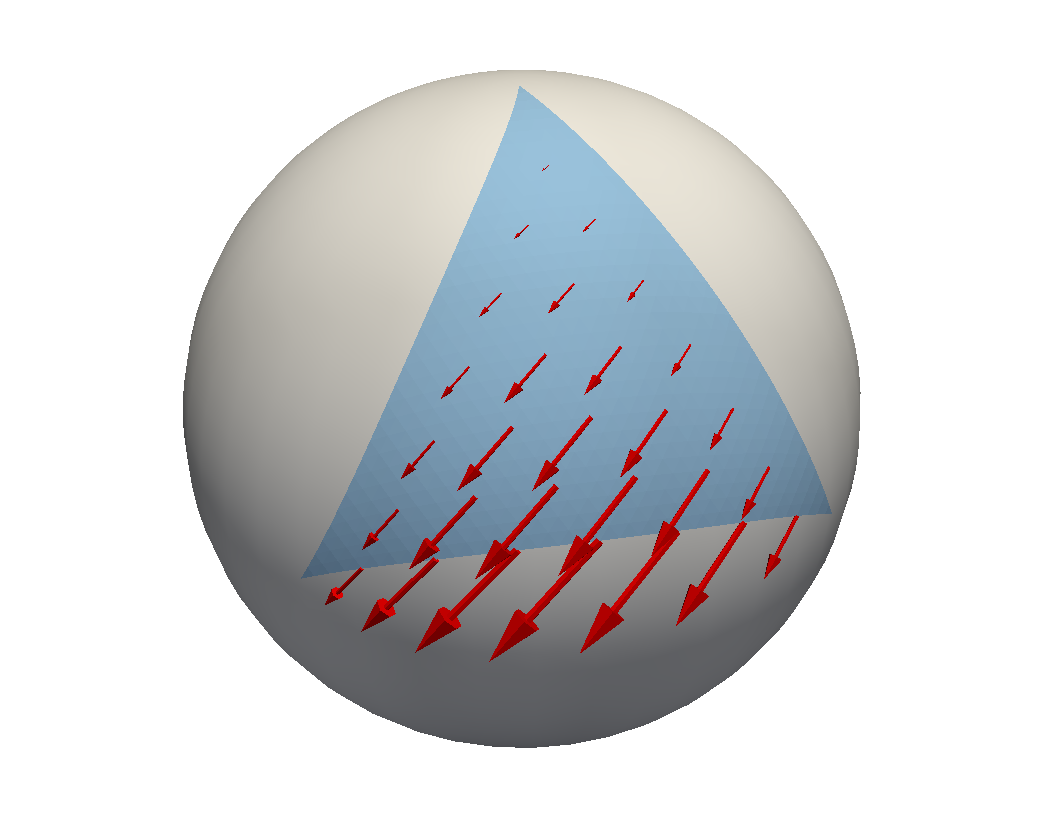}
    \hspace{0.03\textwidth}
    \includegraphics[width=0.39\textwidth, clip, trim= 150 25 150 70]{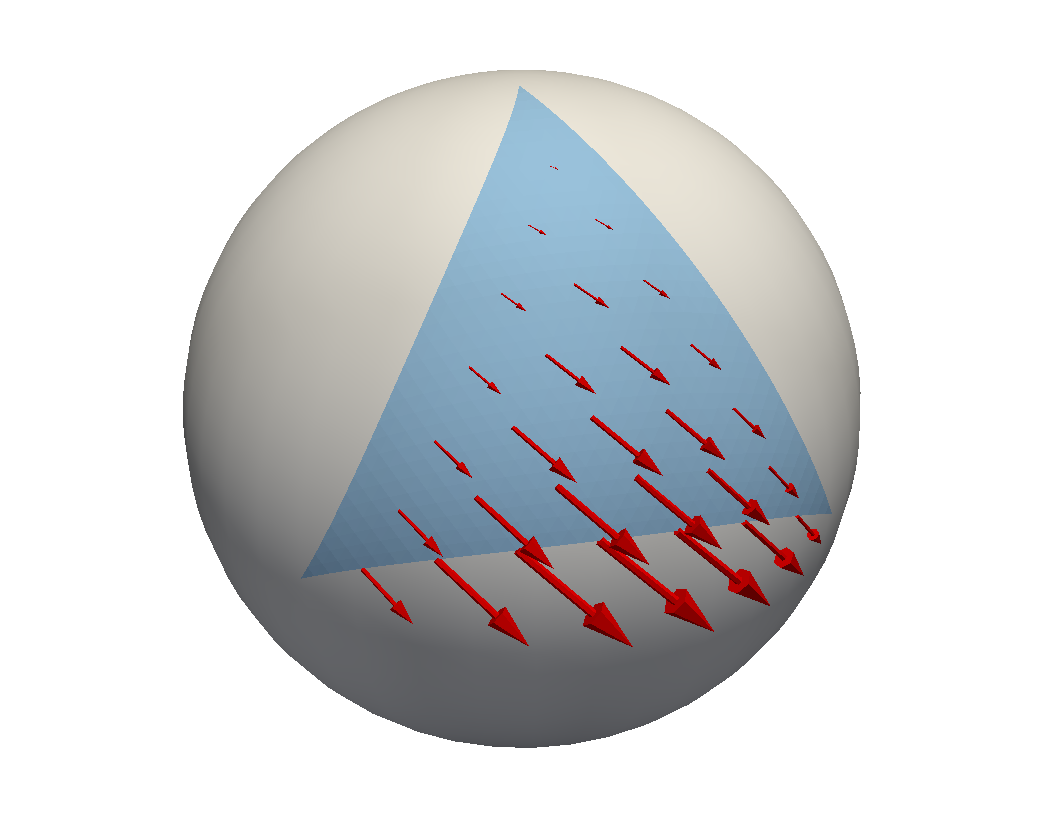}
   \subcaption{Second-order, edge degree of freedom}
   \end{center}
  \end{minipage}
 \caption{Test functions along geodesic interpolation functions from a triangle into the sphere $S^2$.
   These vector fields
   correspond to the shape functions normally used for Lagrangian finite element methods, because
   they are zero on all but one Lagrange point.
   Note how the second-order vertex vector fields in the second row partially point ``backwards'', because the corresponding
   scalar shape functions have negative values on parts of their domains.
 }
 \label{fig:variation_vector_fields}
\end{figure}

Motivated by the previous discussion, we construct test functions as tangent vectors to the GFE manifold
$V_h^M$.  We use the fact that a tangent vector $\eta$ at a given point $u_h \in V_h^M$ is the
derivative of a differentiable curve through $u_h$.

As for the construction of the GFE
spaces themselves, we construct test functions first for interpolation functions on a reference element,
and then piece them together to form global spaces.

\filbreak

\begin{definition}
\label{def:variations_of_geodesic_interpolation}
Let $\Upsilon \in P_p(M)$ be a geometric interpolation function for the values $v_1, \dots, v_m \in M$.
A vector field $\eta$ along $\Upsilon$ is called
a variation of $\Upsilon$, if there is a differentiable curve $c : (-\epsilon, \epsilon) \to P_p(M)$
such that
\begin{equation*}
 \dot{c}|_{t=0}(\xi) = \eta(\xi)
\end{equation*}
for all $\xi \in T_\text{ref}$.
\end{definition}

Figure~\ref{fig:variation_vector_fields} shows six such vector fields along functions
from $P_1^\text{ge}(S^2)$ and $P_2^\text{ge}(S^2)$, mapping the reference triangle to the unit sphere.

Definition~\ref{def:variations_of_geodesic_interpolation} generalizes the well-known
Jacobi fields, because if $T_\text{ref}$ is one-dimensional, then a first-order geodesic
interpolation $\Upsilon^\text{ge}$ is a geodesic curve between the Lagrange values
\cite[Lem.\,2.2]{sander:2012}, and the variations constructed in Definition~\ref{def:variations_of_geodesic_interpolation}
are Jacobi fields~\cite[Thm.\,5.2.1]{jost:2011}.  If, on the other hand, we set $M = \R$,
the standard Lagrangian shape functions are obtained.  Hence, Definition~\ref{def:variations_of_geodesic_interpolation}
is a direct generalization of the test functions normally used in the finite element method.

\bigskip

GFE test functions have an algebraic representation.  Unlike for GFE functions themselves, where the
relationship between discrete functions and algebraic representations is complicated, for test functions
the two are isomorphic again.
\begin{lemma}
\label{lem:vector_isomorphism}
 Let $\Upsilon$ be a geometric interpolation function for the values $v_1,\dots,v_m \in M$.  The generalized Jacobi fields
 on $\Upsilon$ form a vector space $\mathcal{V}_\Upsilon$, which is isomorphic to $\prod_{i=1}^m T_{v_i} M$.
 The isomorphism has an explicit representation as
 \begin{align*}
  \mathcal{T} & \; : \prod_{i=1}^m T_{v_i} M \to \mathcal{V}_\Upsilon \\
  \mathcal{T}[b_1,\dots,b_m](\xi)
       & = \sum_{i=1}^m \frac{\partial \Upsilon(v_1,\dots,v_m;\xi)}{\partial v_i} \cdot b_i
 \end{align*}
 for all $b_i \in T_{v_i}M$, $i=1,\dots,m$, and $\xi \in T_\text{ref}$.
\end{lemma}
\begin{proof}
Let $\eta$ be a generalized Jacobi field along $\Upsilon$.  By definition, there is a curve $c$
in $P_p(M)$ such that $\dot{c}(0) = \eta$.  Let $C : (-\epsilon, \epsilon) \to M^m$ be the
corresponding curve of point values at the Lagrange nodes.  Then, the value of $\eta$ at a point $\xi \in T_\text{ref}$
is
\begin{align*}
 \eta(\xi) & = \dot c(\xi)\big|_{t=0} \\
 & =
 \frac{d}{dt} \Upsilon(C(t), \xi)\Big|_{t=0} \\
 & =
 \sum_{i=1}^m \frac{d}{dv_i} \Upsilon(C(0),\xi) \cdot \frac{d C_i(t)}{dt}\bigg|_{t=0} \\
 & =
 \mathcal{T} [b_1, \dots, b_m](\xi),
\end{align*}
with $b_i = \frac{d C_i(t)}{dt}\big|_{t=0}$, $i = 1,\dots,m$.  Hence there is a set of tangent vectors
$b_1, \dots, b_m$ such that $\mathcal{T} [b_1, \dots, b_m] = \eta$, and therefore $\mathcal{T}$ is surjective.

On the other hand, the restriction of the vector field $\mathcal{T}[b_1,\dots,b_m]$ to the Lagrange nodes $a_1, \dots, a_m$ yields
$b_1,\dots, b_m$, since $\frac{\partial \Upsilon(a_j)}{\partial v_i}$ is the identity if $i=j$,
and zero otherwise.   Therefore $\mathcal{T}$ is injective.

Finally, the linearity of $\mathcal{T}$ follows from the linearity of the derivatives
$\frac{\partial \Upsilon(v_1,\dots,v_m;\xi)}{\partial v_i}$.
\end{proof}

The quantities $\frac{\partial \Upsilon}{\partial v_i}$ appearing in the expression for the isomorphism $\mathcal{T}$
between coefficient vectors and vector fields are not new.  They already appear in the
expressions for the derivatives of energy functionals $J : M^n \to \R$ with respect to coefficients; for gradient descent and
Newton-type methods.  Computation is straightforward if $\Upsilon$ is projection-based interpolation
with an explicitly given projection $\mathcal{P}$.  For geodesic finite elements, it was later argued
in~\cite{sander:2015,sander_neff_birsan:2016} that energy gradients and Hessians are best
evaluated with an automatic differentiation system (and the same argument holds for projection-based finite elements
as well).  But a simple way to evaluate $\frac{\partial \Upsilon}{\partial v_i}$ for geodesic interpolation
has been proposed in the literature nevertheless~\cite{sander:2012,sander:2013}, which we briefly revisit here.

Let $\Upsilon^\text{ge}(v,\xi) : M^m \times T_\text{ref} \to M$ be a function given by geodesic
interpolation, $v_i \in M$, $i=1,\dots,m$ the coefficients corresponding to the $m$ Lagrange nodes, and let
$\xi \in T_\text{ref}$ be arbitrary but fixed. We want to compute the derivatives
\begin{equation*}
 \parder{}{v_i} \Upsilon^\text{ge}(v_1,\dots,v_m; \xi)
 \; : \;
 T_{v_i} M \to T_{\Upsilon^\text{ge}(v,\xi)} M
\end{equation*}
for all $i=1,\dots,m$.  For this, we recall that values $q^*$ of $\Upsilon^\text{ge}$ are minimizers of the functional
\begin{equation*}
 f_{v,\xi} : M \to \R
 \qquad
 f_{v, \xi} (q) \colonequals \sum_{i=1}^m \varphi_i(\xi) \dist(v_i, q)^2.
\end{equation*}
Hence, they fulfill the first-order optimality condition
\begin{equation*}
 F(v_1,\dots,v_m; \xi, q^*)
 \colonequals
 \frac{\partial f_{v,\xi}(q)}{\partial q} \bigg|_{q = q^*}
 = 0.
\end{equation*}
Taking the derivative of this with respect to any of the $v_i$ gives, by the chain rule,
\begin{equation*}
 \frac{dF}{dv_i}
 =
 \parder{F}{v_i} + \parder{F}{q} \cdot \parder{\Upsilon^\text{ge}(v,\xi)}{v_i}
 = 0,
\end{equation*}
with
\begin{equation}
\label{eq:dFdv}
 \parder{F}{v_i} = \varphi_i(\xi) \parder{}{v_i}\parder{}{q} \dist(v_i,q)^2
\end{equation}
and
\begin{equation}
\label{eq:dFdq}
 \frac{\partial F}{\partial q}
 =
 \operatorname{Hess} f_{v,\xi}
 =
 \sum_{i=1}^m \varphi_i(\xi) \frac{\partial^2}{\partial q^2} \dist(v_i, q)^2.
\end{equation}
By \cite[Lemma~3.11]{sander:2015} the matrix $\partial F / \partial q$ is invertible.
Hence the derivative $\frac{\partial \Upsilon^\text{ge}}{\partial v_i}$ of $\Upsilon^\text{ge}(v_1,\dots,v_m;\xi)$ with respect to one
of its coefficients $v_i$ can be computed as a minimization problem to obtain
the value $\Upsilon^\text{ge}(v,\xi)$
and the solution of the linear system of equations
\begin{equation*}
 \parder{F}{q} \cdot \parder{}{v_i} \Upsilon^\text{ge} (v,\xi)
  =
 - \parder{F}{v_i}.
\end{equation*}
The expressions $\parder{}{v_i}\parder{}{q} \dist(v_i,q)^2$ and $\frac{\partial^2}{\partial q^2} \dist(v_i,q)^2$
that appear in \eqref{eq:dFdv} and \eqref{eq:dFdq}, respectively,
encode the geometry of $M$.  Closed-form expressions for both are given in~\cite{sander:2012} for the case
of $M$ being the unit sphere.
For $M = \text{SO(3)}$, the second derivative of $\dist(v,\cdot)^2$ with respect to the second argument
has been computed in~\cite{sander_neff_birsan:2016}.

\subsection{Generalized Jacobi fields as interpolation in the tangent bundle}

The isomorphism $\mathcal{T}$ defined in Lemma~\ref{lem:vector_isomorphism} constructs a vector field along
a given function $\Upsilon$ from a set of tangent vectors $b_1, \dots, b_m$ at the Lagrange nodes.
It can therefore also be interpreted as an interpolation operator for
vector fields.  Using that the tangent bundle $TM$ can be given the structure of a smooth
manifold itself, \citeauthor{hardering:2015} showed the elegant result
that for variations of geodesic interpolation functions, $\mathcal{T}$ can even itself be interpreted
as geodesic interpolation in the sense of Section~\ref{sec:geodesic_interpolation}, if the metric on $TM$ is
chosen appropriately.

The following is taken from \cite{hardering:2015}, Sections~1.2.2 and Remark~2.26.
Let $\pi : TM \to M$ denote the canonical projection.  The tangent space of $TM$ at any point $(q,V)$
splits into the horizontal and the vertical subspace
\begin{equation*}
 T_{(q,V)}TM
 =
 H_{(q,V)} \oplus V_{(q,V)},
\end{equation*}
where the vertical subspace is defined as the kernel of $d\pi_{(q,V)}$.  Roughly speaking, $H_{(q,V)}$ contains
the variations of $q$, and $V_{(q,V)}$ contains the variations of $V$.

For any vector $W \in T_qM$ there exists a unique vector $W^h(q,V) \in H_{(q,V)}$ such that $d\pi(W^h) = W$.
This vector $W^h$ is called the horizontal lift of $W$~\cite{gallot_hulin_lafontaine:2004}, and this lifting defines
an isomorphism between $T_q M$ and $H_{(q,V)}$.  Likewise, there is a vertical lift $W^v$ of $W$ to $V_{(q,V)}$.
For any smooth real-valued function $f$ on $M$, let $df \in T^*_qM$ be its differential at $q$, and $Wf$
the derivative of $f$ in the direction of $W$.  There is a unique vector $W^v(q,V) \in V_{(q,V)}$
such that $W^v(df) = Wf$ for all functions $f$ on $M$. This vertical lift
defines an isomorphism between the vector space $T_qM$ and $V_{(q,V)}$.

Using these concepts, we can define a pseudo-metric $g^h$ on $TM$ from the metric $g$ of $M$.
Let $(q,V)$ be a point on $TM$.  Let $X_1$, $X_2$, $Y_1$, $Y_2$ be elements of $T_qM$, and
$\widetilde{X} = X^h_1 + X^v_2$ and $\widetilde Y = Y^h_1 + Y^v_2$ be elements
of $T_{(q,V)}TM$.  The horizontal lift $g^h$ of $g$ on $TM$ evaluated for $\widetilde{X}$ and $\widetilde{Y}$ is
\begin{equation*}
 g^h_{(q,V)}(\widetilde{X}, \widetilde{Y}) = g_q(X_1^h, Y_2^v) + g_q(X_2^v,Y_1^h).
\end{equation*}
It is a pseudo-Riemannian metric on $TM$ of signature $(k,k)$, with $k$ the dimension of $M$~\cite{kowalski_sekizawa:1988}.

Using this apparatus we can show that geodesic vector field interpolation, originally defined by variation
of geodesic interpolants, can also be
seen as a variational form of geodesic interpolation on $TM$ with respect to the horizontal lift metric.
We do not obtain a minimization formulation of geodesic vector field interpolation, as $g^h$ is only
a pseudo-metric.
\begin{lemma}[\cite{hardering:2015}]
If $(v_i, V^i)$ denotes values in $TM$, $\Upsilon^\text{ge}(v,\cdot)$ the geodesic interpolation of the $v_i$ in $M$,
and $\mathcal{T} = \mathcal{T}[V_1, \dots, V_m]$ the interpolation of the $V_i$ in the sense of
Lemma~\ref{lem:vector_isomorphism}, then we have
\begin{equation*}
 \sum_{i=1}^m \varphi_i(\xi) \log^h_{(\Upsilon^\text{ge}(v,\xi), \mathcal{T}(\xi))} (v_i, V_i)
 =
 (0,0)
 \in T_{(\Upsilon^\text{ge}(\xi),\mathcal{T}(\xi))}TM,
\end{equation*}
where $\log^h_{(q,V)}$ is the inverse of the exponential map of the metric $g^h$ at the point $(q,V)$.
\end{lemma}
This corresponds to the first-order optimality condition~\eqref{eq:first_order_optimality} of geodesic interpolation
in the tangent bundle $TM$.  If only the projection onto the first factor $q$ is considered the formula
degenerates to geodesic interpolation on $M$.

\subsection{Global test function spaces}

We now give discretizations of the global test function spaces.
They are constructed by piecing together local variations continuously across element boundaries.

\begin{definition}
 Let $u_h \in V_h^M$ be a geometric finite element function.  A test function $v_h$ of $u_h$ is a continuous vector field
 along $u_h$ such that $v_h|_T$ is a (generalized) Jacobi field on $u_h|_T$ for all elements $T$ of $\mathcal{G}$.
\end{definition}
We note that this definition is equivalent to saying that a test function $v_h$ of $u_h$ is the derivative
of a curve in $V_h^M$ at $u_h$.  Therefore, test functions are vectors in the tangent space $T_{u_h} V_h^M$.
Indeed, we have:
\begin{lemma}
 Let $u_h \in V_h^M(G)$ be a geometric finite element function, and let
$\bar{u} \in M^n$ be its values at the grid vertices.  Then $T_{u_h} V_h^M$
is isomorphic to $\prod_{i=1}^n T_{u_i}M$.
\end{lemma}
The operator $\mathcal{T}$ defined in Lemma~\ref{lem:vector_isomorphism} extends a set of vectors at the
nodal values $v_1,\dots, v_m$ to a vector field along the interpolation function $\Upsilon$.  Given a GFE function $u_h$,
the local operator $\mathcal{T}$ can be generalized naturally to an operator that maps a set of tangent
vectors at the nodal values of $u_h$ to a test function along $u_h$.  In an abuse of notation, we will
denote both operators by the same letter $\mathcal{T}$.

Evaluation of global GFE test functions is straightforward.  Let $\eta_h$ be such a function for $u_h$,
and let $x \in \Omega$.  Then, to compute $\eta_h(x)$, suppose that $T$ is a grid element with $x \in T$.
Then, if $\xi$ are the local coordinates of $x$ in $T$,
\begin{equation*}
 \eta_h(x) = \mathcal{T}[b_1,\dots,b_m](\xi),
\end{equation*}
where $b_1,\dots,b_m$ are the values of $\eta_h$ at the Lagrange nodes of $T$.  No additional transformation
is necessary.

Finally, let $u_h$ be a fixed GFE function.  As the space $T_{u_h} V_h^M$ of test functions along $u_h$ is a linear space,
it admits a basis representation.  In particular, we can even construct a generalization of the nodal basis
for $T_{u_h} V_h^M$. Let $T_{u_1} M, \dots, T_{u_n} M$ be the set of tangent spaces
at the nodal values of $u_h$.
For each of the spaces $T_{u_1} M, \dots, T_{u_n} M$ select an orthonormal basis, and call the basis vectors
$\phi_{ij}$, $i=1,\dots,n$, $j=1,\dots,\dim M$.  To each vector $\phi_{ij}  \in T_{u_i}M$ corresponds a test function
in $T_{u_h} V_h^M$, defined as the unique function $\Phi_{ij}$ in $T_{u_h} V_h^M$ that is equal to $\phi_{ij}$
at Lagrange node $i$, and equal to the zero vector on all other Lagrange nodes.
The set of all these functions $\Phi_{ij}$, $i=1,\dots,n$, $j=1,\dots, \dim M$, forms a basis of
$T_{u_h} V_h^M$, which we call the  nodal basis.  All six test functions shown in Figure~\ref{fig:variation_vector_fields}
are such nodal basis functions.

\section{Equivalence in the case of minimization problems}
\label{sec:minimization}

In the original presentation of geodesic finite element functions \cite{sander:2012}, only PDEs with a minimization
formulation were considered.  The definition of discrete test function spaces was avoided by moving directly to
an algebraic minimization problem.  First and second variations were hence only ever
taken in the algebraic setting, where it was clear that $\prod_{i=1}^n T_{v_i}M$
is the proper space of variations around an algebraic configuration $\bar{v} \in \mathcal{M} \subset M^n$,
$\bar{v} = (v_1, \dots, v_n)$.

Now that discrete test function spaces are available, it is possible to also consider optimality conditions for minimization problems
in the spaces of GFE functions.  In this section we show that, with the definition of test functions given
in this paper, these two approaches are equivalent.  In other words, we show that the following
diagram commutes:
\begin{center}
\begin{tikzpicture}[node distance=2cm, auto]
  \node (P) {$V_h^M$};
  \node (B) [right of=P] {$T V_h^M$};
  \node (A) [below of=P] {$M^n$};
  \node (C) [below of=B] {$TM^n$};
  \draw[->] (P) to node {$d$} (B);
  \draw[->] (P) to node [swap] {$\mathcal{E}$} (A);
  \draw[->] (A) to node [swap] {$d$} (C);
  \draw[->] (B) to node {$\mathcal{T}^{-1}$} (C);
\end{tikzpicture}
\end{center}
The proof is short and the result may be obvious for people skilled in geometric analysis.  We show it
nevertheless for readers with other backgrounds.

\bigskip

Let $\mathcal{J} : H^1(\Omega,M) \to \R$ be a sufficiently smooth energy functional.  We want to find approximate minimizers
of this functional using the geometric finite element method.  For this, we first restrict
$\mathcal{J}$ to the discrete geometric finite element function space $V_h^M$.
This is possible without any approximation error as the GFE spaces are subsets of $H^1$ (Theorem~\ref{thm:conformity}).

A this point, there are two ways to proceed.  The approach chosen in~\cite{sander:2012} identifies the
discrete space $V_h^M$ with the algebraic space $ \mathcal{M} \subset M^n$ (locally).  This leads to the
algebraic minimization problem in $\mathcal{M}$ (locally) for the functional
\begin{equation*}
 J : \mathcal{M} \to \mathbb{R},
 \qquad
 J(\bar{v}) \colonequals \mathcal{J}(\mathcal{E}^{-1}(\bar{v})).
\end{equation*}

A necessary condition for $\bar{v} \in \mathcal{M}$ to be a minimizer of $J$ is that the
first variation of $J$ at $\bar{v}$ vanishes
\begin{equation}
\label{eq:algebraic_optimality_condition}
 dJ[\bar{v}](\bar{\eta}) = 0
\qquad
\forall \bar{\eta} \in T_{\bar{v}}M^n.
\end{equation}
This is the path through the lower left corner in the diagram above.

The second formulation uses discrete test functions to state the first-order optimality conditions directly
in the discrete function space.
A necessary condition for $v_h$ to be a local minimizer of $\mathcal{J}$ in $V_h^M$
is that the first variation vanishes.
Suppose that the test functions we have defined in Section~\ref{sec:test_function_spaces} are
defined in the correct way.  Then, a necessary condition for $v_h$ to be a local minimizer of $\mathcal{J}$ in $V_h^M$
is that the derivative of $\mathcal{J}$ is zero in the direction of all test functions $\eta_h$
\begin{equation*}
 d\mathcal{J}[v_h](\eta_h) = 0
\qquad
\forall \eta_h \in T_{v_h} V_h^M.
\end{equation*}
Using the evaluation operators $\mathcal{E}$ and $\mathcal{T}^{-1}$ for GFE functions and test functions, respectively,
we can obtain an algebraic form of this optimality condition
\begin{equation}
 \label{eq:algebraic_weak_formulation}
 d \mathcal{J}(\mathcal{E}^{-1}(\bar{v})](\mathcal{T}(\bar{\eta})) = 0,
 \qquad
\forall \bar{\eta} \in T_{\bar{v}}M^n.
\end{equation}
This is the path through the upper right corner in the diagram above.

The following result states that both paths are equivalent.  This is the justification for our definition
of test functions.

\begin{theorem}
 Problems~\eqref{eq:algebraic_optimality_condition} and \eqref{eq:algebraic_weak_formulation}
 are equivalent.
\end{theorem}
\begin{proof}
 Suppose that $\bar{v} \in M^n$ is such that
 \begin{equation*}
  dJ[\bar{v}](\bar{\eta}) = 0
 \end{equation*}
 for any $\bar{\eta} \in T_{\bar{v}} M^n$.  This means that for any $\bar{\eta}$ there is a curve
 $C : (- \epsilon, \epsilon) \to M^n$ with $\dot{C}(0) = \bar{\eta}$ such that
 \begin{equation*}
  \frac{d}{dt} \mathcal{J}(\mathcal{E}^{-1}(C(t)))\Big|_{t=0} = 0.
 \end{equation*}
 Let $c \colonequals \mathcal{E}^{-1}(C(\cdot))$ be the corresponding curve in $V_h^M$.
 Then
 \begin{align*}
  0 & =
  \frac{d}{dt}\mathcal{J}(c(t))\Big|_{t=0} \\
  & =
  d \mathcal{J} (c(0)) \Big[ \frac{dc}{dt}\Big|_{t=0}\Big] \\
  & =
  d \mathcal{J} (v_h) [\eta_h] \\
  & =
  d \mathcal{J}(\mathcal{E}^{-1}(\bar{v})](\mathcal{T}(\bar{\eta})),
 \end{align*}
 which is~\eqref{eq:algebraic_weak_formulation}.
 As the same argument also works backwards, both formulations are equivalent.
\end{proof}

\bibliographystyle{abbrvnat}
\bibliography{paper-sander-gfe-testfunctions}

\end{document}